\documentclass[12pt]{amsart}
\usepackage{amssymb}
\usepackage{verbatim}
\usepackage{hyperref}
\usepackage{color}

\begin{document}

\newtheorem{theorem}{Theorem}    
\newtheorem{proposition}[theorem]{Proposition}
\newtheorem{conjecture}[theorem]{Conjecture}
\def\theconjecture{\unskip}
\newtheorem{corollary}[theorem]{Corollary}
\newtheorem{lemma}[theorem]{Lemma}
\newtheorem{sublemma}[theorem]{Sublemma}
\newtheorem{fact}[theorem]{Fact}
\newtheorem{observation}[theorem]{Observation}
\theoremstyle{definition}
\newtheorem{definition}{Definition}
\newtheorem{notation}[definition]{Notation}
\newtheorem{remark}[definition]{Remark}
\newtheorem{question}[definition]{Question}
\newtheorem{questions}[definition]{Questions}
\newtheorem{example}[definition]{Example}
\newtheorem{problem}[definition]{Problem}
\newtheorem{exercise}[definition]{Exercise}

\numberwithin{theorem}{section}
\numberwithin{definition}{section}
\numberwithin{equation}{section}

\def\reals{{\mathbb R}}
\def\torus{{\mathbb T}}
\def\integers{{\mathbb Z}}
\def\naturals{{\mathbb N}}
\def\complex{{\mathbb C}\/}
\def\distance{\operatorname{distance}\,}
\def\support{\operatorname{support}\,}
\def\dist{\operatorname{dist}\,}
\def\Span{\operatorname{span}\,}
\def\degree{\operatorname{degree}\,}
\def\dim{\operatorname{dim}\,}
\def\codim{\operatorname{codim}}
\def\trace{\operatorname{trace\,}}
\def\Span{\operatorname{span}\,}
\def\dimension{\operatorname{dimension}\,}
\def\codimension{\operatorname{codimension}\,}
\def\nullspace{\scriptk}
\def\kernel{\operatorname{Ker}}
\def\Re{\operatorname{Re\,} }
\def\Im{\operatorname{Im\,} }
\def\eps{\varepsilon}
\def\lt{L^2}
\def\diver{\operatorname{div}}
\def\curl{\operatorname{curl}}
\def\expect{\mathbb E}
\def\bull{$\bullet$\ }
\def\det{\operatorname{det}}
\def\Det{\operatorname{Det}}

\newcommand{\norm}[1]{ \|  #1 \|}
\newcommand{\Norm}[1]{ \Big\|  #1 \Big\| }
\newcommand{\set}[1]{ \left\{ #1 \right\} }
\def\one{{\mathbf 1}}
\newcommand{\modulo}[2]{[#1]_{#2}}
\newcommand{\abr}[1]{ \langle  #1 \rangle}

\def\bp{\mathbf p}
\def\bff{\mathbf f}
\def\bg{\mathbf g}

\def\scriptf{{\mathcal F}}
\def\scriptg{{\mathcal G}}
\def\scriptm{{\mathcal M}}
\def\scriptb{{\mathcal B}}
\def\scriptc{{\mathcal C}}
\def\scriptt{{\mathcal T}}
\def\scripti{{\mathcal I}}
\def\scripte{{\mathcal E}}
\def\scriptv{{\mathcal V}}
\def\scriptw{{\mathcal W}}
\def\scriptu{{\mathcal U}}
\def\scriptS{{\mathcal S}}
\def\scripta{{\mathcal A}}
\def\scriptr{{\mathcal R}}
\def\scripto{{\mathcal O}}
\def\scripth{{\mathcal H}}
\def\scriptd{{\mathcal D}}
\def\scriptl{{\mathcal L}}
\def\scriptn{{\mathcal N}}
\def\scriptp{{\mathcal P}}
\def\scriptk{{\mathcal K}}
\def\scriptP{{\mathcal P}}
\def\scriptj{{\mathcal J}}
\def\frakv{{\mathfrak V}}
\def\frakG{{\mathfrak G}}
\def\frakA{{\mathfrak A}}

\author{Marcos Charalambides}
\address{
        Marcos Charalambides\\
        Department of Mathematics\\
        University of California \\
        Berkeley, CA 94720-3840, USA}
\email{marcos@math.berkeley.edu}

\thanks{Research supported in part by NSF grant DMS-0901569.}

\date{}

\title
[On Restricting Cauchy-Pexider Equations to Submanifolds]
{On Restricting Cauchy-Pexider \\ Equations to Submanifolds}

\begin{abstract}
Sufficient geometric conditions are given which determine when the Cauchy-Pexider functional equation $f(x)g(y)=h(x+y)$ restricted to $x,y$ lying on a hypersurface in $\reals^d$ has only solutions which extend uniquely to exponential affine functions $\reals^d\to\complex$ (when $f$, $g$, $h$ are assumed to be measurable and non-trivial). The Cauchy-Pexider-type functional equations $\prod_{j=0}^df_j(x_j)=F(\sum_{j=0}^dx_j)$ for $x_0,\ldots,x_d$ lying on a curve and $f_1(x_1)f_2(x_2)f_3(x_3)=F(x_1+x_2+x_3)$ for $x_1,x_2,x_3$ lying on a hypersurface are also considered.
\end{abstract}

\maketitle

\section{Introduction}

The restricted Cauchy-Pexider functional equation is
\begin{equation}\label{Cauchy}
f(x)g(y)=h(x+y) \text{ for all } (x,y)\in S
\end{equation}
 where $S$ is some subset of $\reals^d\times\reals^d$, $d\ge 1$ and $f$, $g$ and $h$ are complex-valued functions defined on suitable domains.

 The case when (\ref{Cauchy}) holds for $S=\reals^d\times\reals^d$ is the classical Cauchy-Pexider equation; see \cite{aczel} for an introduction to this and other related functional equations. In this case, if $f$, $g$, $h$ are measurable functions which vanish only on sets of measure zero then there exist $v\in\complex^d$ and $c_1,c_2\in\complex$ such that $f(x) = c_1e^{v\cdot x}$ and $g(x) = c_2e^{v\cdot x}$ for Lebesgue-almost every $x\in\reals^d$.

For $A,B\subset\reals^d$ and $n\in \mathbb{N}$, define the sumsets $A+B=\{a+b\,\,|\,\,a\in A, b\in B\}$ and $nA=\{a_1+\ldots+a_n\,\,|\,\,a_1,\ldots,a_n\in A\}$. Write $v\cdot x$ for the sum $\sum v_jx_j$ whenever $v\in\complex^d$ and $x\in\reals^d$. Denote the open ball of radius $r$ and centre $c$ in $\reals^d$ by $B_r(c)$. If $U$ is a set in an underlying topological space (which will be clear from context), write $\overline{U}$ for its closure. If $(X,\mathcal{M},\mu)$ is a measure space, we will say that $S$ is $\mu$-null (respectively $\mu$-full) if $S\in\mathcal{M}$ and $\mu(S)=0$ (respectively $\mu(X\setminus S)=0$).

Consider the unit sphere $\mathbb{S}^2\subset\reals^3$ with induced surface measure $\sigma$. In Christ-Shao \cite{christshao} it is shown that in the case when (\ref{Cauchy}) holds for a $\sigma^2$-full set $S\subset\mathbb{S}^2\times\mathbb{S}^2$,  where $f:\mathbb{S}^2\to\complex$, $g:\mathbb{S}^2\to\complex$, $h:\overline{B_2(0)}\to\complex$ are measurable and $f$ and $g$ vanish only on $\sigma$-null sets, it follows that $f$ and $g$ must be of the form $f(x) = c_1e^{v\cdot x}$, $g(x) = c_2e^{v\cdot x}$ for $\sigma$-almost every $x\in\mathbb{S}^2$. 

This phenomenon can fail. Consider first the case of the parabola $P\subset\reals^2$ consisting of all points of the form $(x,x^2)$ for $x\in\reals$. Then, given any non-vanishing measurable function $s:\reals\to\reals$, the functions $f$, $g$, $h$ defined respectively on $P$, $P$ and $2P$ by $f(x) = g(x) = s(x)$ and $h(x+y)=s(x)s(y)$ are measurable and satisfy the functional equation \[f(x)g(y)=h(x+y)\text{ for all }(x,y)\in P^2.\] Indeed, the only issue is whether $h$ is well-defined. But it is straightforward to check that each $z\in 2P$ has a unique, up to order, decomposition as $z=x+y$ for $x,y\in P$. This example can be used to obtain higher-dimensional counterexamples. For instance, consider the hypersurface $C=P\times\reals\subset\reals^3$. Given any non-vanishing measurable function $s:\mathbb{R}\to\mathbb{R}$, the functions $f$, $g$, $h$ defined respectively on $C$, $C$ and $2C$ by $f({\bf x} )=g({\bf x})=s(x_2)$ and $h({\bf x}+{\bf y})=f({\bf x})g({\bf y})=s(x_2)s(y_2)$ are measurable and satisfy the functional equation (\ref{Cauchy}) with $S=C\times C$. 

In the sequel, submanifolds of $\reals^d$ will always be smooth embedded submanfolds of $\reals^d$ (but we do not assume that they are compact or closed) and we will refer to a codimension $1$ submanifold of $\reals^d$ as a hypersurface. A submanifold $M$ of $\reals^d$ comes equipped with the induced measure $\mu$ associated to the Riemannian structure on $M$ induced by the standard Euclidean structure on $\reals^d$. To the sumset $nM\subset\reals^d$ we associate the $n$-fold convolution measure $\mu *\cdots *\mu$ which is given by \[\mu *\cdots *\mu(E)=\mu^n(\{(x_1,\ldots,x_n)\in M^n\,\,|\,\,x_1+\cdots+x_n\in E\}).\]

For $d\ge 3$, a \emph{cylinder} is defined to be a hypersurface which is, up to rigid motions, an open subset of $\Gamma\times\mathbb{R}^{d-2}$ for some smooth embedded plane curve $\Gamma\subset\mathbb{R}^2$. 

Let $M$ be a hypersurface with induced measure $\mu$. The main result of this paper is that, unless $M$ is flat somewhere or contains a cylinder, whenever the non-vanishing measurable functions $f$, $g$, $h$ satisfy (\ref{Cauchy}) for $\mu^2$-a.e. pair $(x,y)$, it follows that $f$ and $g$ must be of the form $f(x) = c_1e^{v\cdot x}$, $g(x) = c_2e^{v\cdot x}$ $\mu$-almost everywhere and, furthermore, $c_1$, $c_2$, $v$ are uniquely determined.

If $M$ is orientable, let $\mathcal{G}=\mathcal{G}^M:M\to\mathbb{S}^{d-1}$ denote the smooth Gauss normal map, where $\mathbb{S}^{d-1}$ is the unit sphere in $\mathbb{R}^d$. Observe that for any hypersurface $M$, whether the linear map $(d\mathcal{G}^U)_x$ is zero for a point $x\in M$ and some orientable open neighbourhood $U\subset M$ of $x$ is independent of the choice of $U$.

\begin{definition}
A hypersurface $M$ is \emph{nowhere-flat} if for every $x\in M$ there exists an orientable open neighbourhood $U$ of $x$ such that ($d\mathcal{G}^U)_x$ is non-zero. $M$ is \emph{cylinder-free} if it does not contain a cylinder.
\end{definition}

Note that $M$ is nowhere-flat if and only if at each point $x\in M$ at least one of the principal curvatures is non-zero. Hypersurfaces in $\reals^d$ for $d\ge 3$ which are nowhere-flat and cylinder-free include the cone $\{(x,t)\in\reals^{d-1}\times\reals\,\,|\,\,t=\norm{x},\, t>0\}$ and any hypersurface of non-vanishing Gaussian curvature.

\begin{theorem}\label{theorem:mainmult}
Let $d\ge 3$. Let $M\subset \reals^d$ be a connected hypersurface with induced measure $\mu$ which is nowhere-flat and cylinder-free. Suppose that $f,g:M\to\complex$ and $h:2M\to\complex$ are measurable functions satisfying \[f(x)g(y)=h(x+y) \text{ for }\mu^2\text{-almost every } (x,y)\in M^2.\] Suppose further that $f^{-1}(\{0\})$ and $g^{-1}(\{0\})$ are $\mu$-null. Then there exist unique $v\in\complex^d$ and $c_1,c_2\in\complex\backslash\{0\}$ such that $f(x)=c_1\exp(v\cdot x)$, $g(x)=c_2\exp(v\cdot x)$ for $\mu$-almost every $x\in M$.
\end{theorem}

The proof will also establish the analogous result for the additive Cauchy-Pexider functional equation.

\begin{theorem}\label{theorem:mainadd}
Let $d\ge 3$. Let $M\subset \reals^d$ be a connected hypersurface with induced measure $\mu$ which is nowhere-flat and cylinder-free. Suppose that $f,g:M\to\complex$ and $h:2M\to\complex$ are measurable functions satisfying \[f(x)+g(y)=h(x+y) \text{ for }\mu^2\text{-almost every } (x,y)\in M^2.\] Then there exist unique $v\in\complex^d$ and $c_1,c_2\in\complex$ such that $f(x)=v\cdot x + c_1$, $g(x)=v\cdot x + c_2$ for $\mu$-almost every $x\in M$.
\end{theorem}

In \S\ref{section:app}, the following approximate versions of these theorems are proved. 

\begin{theorem}\label{theorem:appmain}
Let $d\ge 3$. Let $M\subset\reals^d$ be a bounded hypersurface with finite induced measure $\mu$ which is nowhere-flat and cylinder-free. Given $\epsilon>0$ there exists $\delta>0$ with the following property. Whenever $f,g:M\to\complex$ are measurable functions vanishing only on a $\mu$-null set and $h:2M\to\complex$ is a measurable function satisfying
\begin{equation}
\mu^2(\{(x,y)\in M^2\,\,|\,\,|f(x)g(y)h(x+y)^{-1}-1|>\delta\})<\delta
\end{equation}
it follows that there exist $c\in\complex$ and $v\in\complex^d$ such that \[\mu(\{|f(x)(c\exp(v\cdot x))^{-1}-1|>\epsilon\})<\epsilon.\]
\end{theorem}

\begin{theorem}
Let $d\ge 3$. Let $M\subset\reals^d$ be a bounded hypersurface with finite induced measure $\mu$ which is nowhere-flat and cylinder-free. Given $\epsilon>0$ there exists $\delta>0$ with the following property. Whenever $f,g:M\to\complex$, $h:2M\to\complex$ are measurable functions satisfying
\begin{equation}
\mu^2(\{(x,y)\in M^2\,\,|\,\,|f(x)+g(y)-h(x+y)|>\delta\})<\delta
\end{equation}
it follows that there exist $c\in\complex$ and $v\in\complex^d$ such that \[\mu(\{|f(x)-(v\cdot x + c)|>\epsilon\})<\epsilon.\]
\end{theorem}

The final sections deal with related restricted functional equations. 

As in the two-dimensional example considered above, for a generic embedded curve $\Gamma\subset\reals^d$ each $z\in d\Gamma$ can be written as \[z=x_1+\ldots+ x_d\] where $x_1,\ldots,x_d\in\Gamma$ and $x_1,\ldots,x_d$ are unique up to order. Thus the $d$-fold version of the Cauchy-Pexider equation restricted to curves will not force solutions to be exponential affine, in general. On the other hand, for a suitably non-degenerate curve, the $(d+1)$-fold version does; this is the content of the following theorems which are proved in \S\ref{section:curves}.

\begin{theorem}\label{theorem:curves}
Let $d\ge 2$. Let $\Gamma\subset \reals^d$ be a connected submanifold of dimension $1$ with induced measure $\gamma$. Suppose that no open subset of $\Gamma$ lies in an affine hyperplane. Suppose that $f_0,\ldots,f_{d}:\Gamma\to\complex$  and $F:(d+1)\Gamma\to\complex$ are measurable functions satisfying \[f_0(x_0)\cdots f_d(x_d)=F(x_0+\ldots+x_d) \text{ for }\gamma^{d+1}\text{-almost every } (x_0,\ldots,x_d)\in \Gamma^{d+1}.\] Suppose further that for each $0\le j\le d$, the set $f_j^{-1}(\{0\})$ is $\gamma$-null. Then there exist unique $v\in\complex^d$ and $c_0,\ldots,c_d\in\complex$ such that for each $0\le j\le d$, $f_j(x)=c_j\exp(v\cdot x)$ for $\gamma$-almost every $x\in \Gamma$.
\end{theorem}

\begin{theorem}\label{theorem:curvesadd}
Let $d\ge 2$. Let $\Gamma\subset \reals^d$ be a connected submanifold of dimension $1$ with induced measure $\gamma$. Suppose that no open subset of $\Gamma$ lies in an affine hyperplane. Suppose that $f_0,\ldots,f_{d}:\Gamma\to\complex$  and $F:(d+1)\Gamma\to\complex$ are measurable functions satisfying \[f_0(x_0)+\ldots +f_d(x_d)=F(x_0+\ldots+x_d) \text{ for }\gamma^{d+1}\text{-almost every } (x_0,\ldots,x_d)\in \Gamma^{d+1}.\] Then there exist unique $v\in\complex^d$ and $c_0,\ldots,c_d\in\complex$ such that for each $0\le j\le d$, $f_j(x)=v\cdot x+c_j$ for $\gamma$-almost every $x\in \Gamma$.
\end{theorem}

In \S\ref{section:threefold}, the two-fold product in the left-hand side of (\ref{Cauchy}) is replaced by a product of three functions; by combining an adaptation of the strategy in \S\ref{section:strategy} with Theorems~\ref{theorem:curves},~\ref{theorem:curvesadd}, we establish the following theorems for hypersurfaces.

\begin{theorem}\label{theorem:threefold}
Let $d\ge 2$. Let $M\subset \reals^d$ be a connected hypersurface with induced measure $\mu$ which is nowhere-flat. Suppose that $f_1,f_2,f_3:M\to\complex$ and $F:3M\to\complex$ are measurable functions satisfying \[f_1(x_1)f_2(x_2)f_3(x_3)=F(x_1+x_2+x_3) \text{ for }\mu^3\text{-almost every } (x_1,x_2,x_3)\in M^3.\] Suppose further that $f_j^{-1}(\{0\})$ is $\mu$-null for $1\le j\le 3$. Then there exist unique $v\in\complex^d$ and $c_1,c_2,c_3\in\complex$ such that for $j=1,2,3$, $f_j(x)=c_j\exp(v\cdot x)$ for $\mu$-almost every $x\in M$.
\end{theorem}

\begin{theorem}
Let $d\ge 2$. Let $M\subset \reals^d$ be a connected hypersurface with induced measure $\mu$ which is nowhere-flat. Suppose that $f_1,f_2,f_3:M\to\complex$ and $F:3M\to\complex$ are measurable functions satisfying \[f_1(x_1)+f_2(x_2)+f_3(x_3)=F(x_1+x_2+x_3) \text{ for }\mu^3\text{-almost every } (x_1,x_2,x_3)\in M^3.\] Then there exist unique $v\in\complex^d$ and $c_1,c_2,c_3\in\complex$ such that for $1\le j\le 3$, $f_j(x)=v\cdot x + c_j$ for $\mu$-almost every $x\in M$.
\end{theorem}

Other than in the aforementioned \cite{christshao}, these functional equations have arisen in Foschi \cite{foschi} (for four specific submanifolds $M$). The proof provided in this paper presents an alternative approach to the methods in \cite{foschi}, one which is less reliant on the rigid geometry of particular $M$.
\newline

\paragraph{\bf{Acknowledgement}}

The author is grateful to M. Christ for several illuminating discussions, especially for those relating to the method explored in \S\ref{section:strategy}.

\section{A General Strategy} \label{section:strategy}

Define $\lambda^n$ to be $n$-dimensional Lebesgue measure on $\reals^n$ and $\lambda=\lambda^1$. By a slight abuse of notation, we will also write $\lambda^n$ for the induced measure on an $n$-dimensional linear subspace of $\reals^d$. 

A starting point for our proofs of Theorems~\ref{theorem:mainmult} and \ref{theorem:appmain} will be the following approximate version for the solution to the Cauchy-Pexider functional equation on open balls of $\reals^d$ which is proved in Christ \cite{christyoungs}.

\begin{lemma}\label{lemma:classical}
Given $\epsilon>0$ there exists $\delta>0$ depending on the dimension $d$ with the following property. Whenever $B\subset\reals^d$ is a non-empty ball and $f:B\to\complex$, $g:B\to\complex$, $h:2B\to\complex$ are measurable functions vanishing only on $\lambda^d$-null sets and satisfying \[\lambda^{2d}(\{(x,y)\in B^2\,\,|\,\,|f(x)g(y)h(x+y)^{-1}-1|>\delta\})<\lambda^d(B)^2\delta,\] it follows that there exist $c\in\complex\backslash\{0\}$ and $v\in\complex^d$ such that \[\lambda^d(\{x\in B\,\,|\,\,|f(x)(c\exp(v\cdot x))^{-1}-1|>\epsilon\})<\lambda^d(B)\epsilon.\]

If, in addition, $f(x)g(y)=h(x+y)$ for $\lambda^{2d}$-a.e. $(x,y)\in B^2$ then there exist $c\in\complex\backslash\{0\}$ and $v\in\complex^d$ such that $f(x)=c\exp(v\cdot x)$ for $\lambda^d$-a.e. $x\in B$.
\end{lemma}

The next lemma will allow us to work locally.

\begin{lemma}\label{lemma:localise}
Suppose that the hypersurface $M$ with induced measure $\mu$ is nowhere-flat. Let $U$, $V$ be open subsets of $M$ with non-empty intersection. Suppose that the function $f:U\cup V\to\complex$ satisfies $f(x)= c\exp(v\cdot x)$ for $\mu$-a.e. $x\in U$ and $f(x)=c_1\exp(v_1\cdot x)$ for $\mu$-a.e. $x\in V$ where $v,v_1\in\complex^d$ and $c,c_1\in\complex\backslash\{0\}$. Then $v=v_1$ and $c=c_1$.
\end{lemma}

\begin{proof}
For $\mu$-almost every $x\in U\cap V$, \[c\exp(v\cdot x)=f(x)=c_1\exp(v_1\cdot x).\] Thus for all such $x$,
\begin{eqnarray}&(\Re v-\Re v_1)\cdot x=a& \\ &(\Im v - \Im v_1)\cdot x = b + 2\pi n(x)&
\end{eqnarray}
where $a,b\in\reals$ satisfy $e^{a+ib}=c_1c^{-1}$ and $n:U\cap V\to \integers$ is some function. Since $M$ is nowhere-flat, no open subset of $M$ can lie on an affine hyperplane. Thus $v=v_1$ so also $c_1c^{-1}=1$.
\end{proof}


In particular, the lemma implies the stated uniqueness in Theorem~\ref{theorem:mainmult}. Furthermore, it will suffice to prove the theorem for some open cover of $M$. Thus it will be no loss of generality to prove Theorem~\ref{theorem:mainmult} for a connected, orientable, bounded hypersurface $M$ with $\mu$ finite. In particular, this means that the Gauss normal map $\mathcal{G}$ is globally defined on $M$. Define also the map $\mathcal{\overline{G}}:M\to\overline{\mathbb{S}^{d-1}}$, where $\overline{\mathbb{S}^{d-1}}$ is the unit sphere with antipodal points identified, which maps $x\in M$ to the class of $\mathcal{G}(x)$. We will identify the tangent space $T_xM$ with the linear hyperplane $\text{span}\{\mathcal{G}(x)\}^\perp\subset\reals^d$ in the canonical way.

In this section, the strategy used in \cite{christshao} for $\mathbb{S}^2\subset\reals^3$ will be generalised to hypersurfaces. The proof of Theorem~\ref{theorem:mainmult} will be completed in the subsequent section.

Let $d\ge 3$ and write points in $(\reals^d)^4\times(\reals^d)^4$ as $(x,y)$ where $x=(x_1,x_2,x_3,x_4)$ and $y=(y_1,y_2,y_3,y_4)$ with $x_j, y_j\in\reals^d$ for $1\le j\le 4$. Let $\Pi$ be the hyperplane in $(\reals^d)^4\times(\reals^d)^4$ defined by $x_1+y_2=x_3+y_4$ and $y_1+x_2=y_3+x_4$.  Let $\mathcal{P}_M=(M^4\times M^4)\cap\Pi$. By a \emph{smooth point} of $\mathcal{P}_M$ we mean a point where $M^4\times M^4$ intersects $\Pi$ transversally. Write $\mathcal{S}_M$ for the set of smooth points of $\mathcal{P}_M$. Then $\mathcal{S}_M$ is a submanifold of $\mathbb{R}^{8d}$ of dimension $8(d-1)+6d-8d = 6d-8$. Write $\sigma$ for the volume measure on $\mathcal{S}_M$ associated to the induced Riemannian structure.

Consider the $3d$-dimensional hyperplane $\Lambda\subset(\mathbb{R}^d)^4$ of points $w=(w_1,w_2,w_3,w_4)$ satisfying $w_1+w_2=w_3+w_4$. The linear addition map $(\mathbb{R}^d)^4\times(\mathbb{R}^d)^4\to(\mathbb{R}^d)^4$, $(x,y)\mapsto x+y$ restricts to a smooth map $\pi_M:\mathcal{S}_M\to\Lambda$. Write $\mathcal{R}_M$ for the set of \emph{regular points} of $\pi_M$, that is the points of $\mathcal{S}_M$ where $\pi_M$ is a submersion. Then $\mathcal{R}_M$ is an open subset of $\mathcal{S}_M$.

\begin{lemma}\label{lemma:geomcond} Let $(x,y)\in \mathcal{P}_M$. If 
\begin{equation}\label{smooth}
\overline{\mathcal{G}}(x_1)\ne\overline{\mathcal{G}}(y_2)\text{ and } \overline{\mathcal{G}}(x_2)\ne\overline{\mathcal{G}}(y_1)
\end{equation} then $(x,y)$ is a smooth point of $\mathcal{P}_M$.

If, in addition, 
\begin{equation}\label{regular1}
\overline{\mathcal{G}}(x_j)\ne\overline{\mathcal{G}}(y_j) \text{ for }1\le j\le 4\end{equation} and
\begin{equation}\label{regular2}
\bigcap_{j=1}^4\text{\emph{span}}\{\mathcal{G}(x_j), \mathcal{G}(y_j)\}=\{0\}\end{equation} then $(x,y)$ is a regular point of $\pi_M$. 
\end{lemma}

\begin{proof}
The point $(x,y)\in\mathcal{P}_M$ is smooth if and only if \[T_{x_1}M+T_{y_2}M-T_{x_3}M-T_{y_4}M=T_{y_1}M+T_{x_2}M-T_{y_3}M-T_{x_4}M=\reals^d.\] Thus condition (\ref{smooth}) is sufficient.

For $(x,y)\in\mathcal{S}_M$, we identify $T_{(x,y)}\mathcal{S}_M$ with $(\prod_{j=1}^4T_{x_j}M\times\prod_{j=1}^4T_{y_j}M)\cap\Pi$ and $T_{x+y}\Lambda$ with $\Lambda$ in the canonical way. The differential $(d\pi_M)_{(x,y)}:T_{(x,y)}\mathcal{S}_M\to T_{x+y}\Lambda$ is then given by \[(d\pi_M)_{(x,y)}(u,v)=u+v\] for $(u,v)=(u_1,u_2,u_3,u_4,v_1,v_2,v_3,v_4)$ with $u_j\in T_{x_j}M$, $v_j\in T_{y_j}M$ for $1\le j\le 4$ and $u_1+v_2=u_3+v_4$, $v_1+u_2=v_3+u_4$. Thus for $(d\pi_M)_{(x,y)}$ to be surjective we certainly need $T_{x_j}M+T_{y_j}M=\reals^d$ for each $1\le j\le 4$. This condition is equivalent to (\ref{regular1}). 

Let $w\in\Lambda$ and assume that (\ref{regular1}) and (\ref{regular2}) hold. It suffices to show that there exist $u_j\in T_{x_j}M$, $v_j\in T_{y_j}M$ for $1\le j\le 4$ such that $u_1+v_2=u_3+v_4$ and $u_j+v_j=w_j$ for $1\le j\le 4$; we then also have \[v_1+u_2=(w_1-u_1)+(w_2-v_2)=(w_3+w_4)-(u_3+v_4)=v_3+u_4.\] 
From the equation $u_j+v_j=w_j$, we can vary $u_j$ freely over a certain translate of the linear space $T_{x_j}M\cap T_{y_j}M =\text{span}\{\mathcal{G}(x_j),\mathcal{G}(y_j)\}^\perp$ and then $v_j=w_j-u_j$ is determined. Alternatively we can vary $v_j$ freely over a certain translate of $\text{span}\{\mathcal{G}(x_j),\mathcal{G}(y_j)\}^\perp$ and then $u_j$ is determined. As $u_1,v_2,u_3,v_4$ vary freely over these affine spaces, condition (\ref{regular2}) implies that $u_1+v_2-u_3-v_4$ varies over all of $\reals^d$. In particular, there exists a choice such that $u_1+v_2=u_3+v_4$.
\end{proof}

 For $S\subset M^2$, let $\mathcal{S}_M(S) = \{(x,y)\in\mathcal{S}_M\,|\, (x_i,y_j)\in S\,\, \text{for all} \,\,1\le i,j\le 4\}$. The above setup is motivated by the following observation.

\begin{lemma}\label{lemma:regroup}
Let $f:M\to\complex$, $g:M\to\complex$ and $h:2M\to\complex$. Suppose that $f(x)g(y)=h(x+y)$ for all $(x,y)\in S\subset M^2$. Whenever $z\in\Lambda$ is in the image of $\mathcal{S}_M(S)$ under $\pi_M$,  \[h(z_1)h(z_2)=h(z_3)h(z_4).\] 
\end{lemma}

\begin{proof}
Consider $z=\pi_M(x,y) = x+y$ for $(x,y)\in\mathcal{S}_M(S)$. By expanding using (\ref{Cauchy}) and regrouping, \[h(z_1)h(z_2)=f(x_1)g(y_1)f(x_2)g(y_2)=h(x_1+y_2)h(y_1+x_2)\] and similarly \[h(z_3)h(z_4)=f(x_3)g(y_3)f(x_4)g(y_4)=h(x_3+y_4)h(y_3+x_4).\] But $x_1+y_2=x_3+y_4$ and $y_1+x_2=y_3+x_4$.
\end{proof}

Define $R_M=\{x+y\,\, |\,\, x,y\in M,\,\, \overline{\mathcal{G}}(x)\ne\overline{\mathcal{G}}(y)\}$.

\begin{lemma}  The set $R_M$ is an open subset of $\reals^d$. If $M$ is nowhere-flat then $R_M$ is dense in $2M$.\end{lemma}

\begin{proof}
The addition map $\alpha:M^2\to\mathbb{R}^d$, $(x,y)\mapsto x+y$ is a submersion at $(x,y)$ if and only if $T_xM+T_yM=\mathbb{R}^d$. So $R_M$ is the image of the regular points of $\alpha$. If $M$ is nowhere-flat then $\mathcal{G}$ is not constant on any open subset of $M$.
\end{proof}

From the proof of Lemma~\ref{lemma:geomcond}, it follows that condition (\ref{regular1}) is also necessary for $(x,y)$ to be a regular point of $\pi_M$. Thus we can define a submersion $\gamma_M:\mathcal{R}_M\to R_M$ given by $(x,y)\mapsto x_1+y_1$. The following proposition describes the main property of $\gamma_M$ we will need; the proof will be given in the next section.

\begin{proposition}\label{prop:gammasurjective}  Let $M$ be nowhere-flat and cylinder-free. Then the submersion $\gamma_M:\mathcal{R}_M\to R_M$ is surjective.\end{proposition}

If $\pi:P\to N$ is smooth and $\rho$ is a measure on $P$, write $\pi_*(\rho)$ for the pushforward of $\rho$ under $\pi$, that is the measure on $N$ defined by $\pi_*(\rho)(E)=\rho(\pi^{-1}(E))$. The following fact will be useful; see, for example, \cite{ponomarev} for a proof.

\begin{lemma}\label{lemma:abscont}
Suppose that $\pi:P\to N$ is a surjective submersion between bounded submanifolds $P\subset\reals^p$ and $N\subset\reals^n$. Let $\rho$ and $\nu$ denote the measures associated to the induced Riemannian structures on $P$ and $N$ respectively. Then $\nu$ and $\pi_*(\rho)$ are mutually absolutely continuous.
\end{lemma}

Observe that for each $1\le i,j\le 4$, the smooth map $\mathcal{R}_M\to M^2$ given by $(x,y)\mapsto(x_i,y_j)$ is a submersion. Thus, by Lemma~\ref{lemma:abscont}, if $S\subset M^2$ is $\mu^2$-full then \begin{equation}\label{Risfull}\sigma(\mathcal{R}_M\backslash\mathcal{R}_M\cap\mathcal{S}_M(S))=0.\end{equation}

Let $M$ be nowhere-flat and cylinder-free. The addition map $M^2\to\reals^d$ restricts to a surjective submersion $\alpha_M:\mathcal{M}\to R_M$ where $\mathcal{M}$ is the open subset of $M^2$ consisting of points $(x,y)$ such that $\overline{\mathcal{G}}(x)\ne\overline{\mathcal{G}}(y)$. The measure $\mu^2|_{\mathcal{M}}$ coincides with the induced measure on $\mathcal{M}$ as a submanifold of $(\reals^d)^2$. Since $\mu^2(M^2\backslash\mathcal{M})=0$, Lemma~\ref{lemma:abscont} implies that 
\begin{equation}\label{convequalsleb}
\lambda^d|_{R_M}\text{ and }(\mu*\mu)|_{R_M}\text{ are mutually absolutely continuous}
\end{equation}
where $\mu*\mu$ denotes the measure given by $\mu*\mu(E)=\mu^2(\{(x,y)\in M^2\,\,|\,\,x+y\in E\})$.

Suppose that $f:M\to\complex$, $g:M\to\complex$ are measurable functions which only vanish on sets of $\mu$-measure zero and $h:2M\to\complex$ is measurable. By (\ref{convequalsleb}), $h$ vanishes only on a set of $\lambda^d$-measure zero. 

Suppose that (\ref{Cauchy}) is satisfied for some $S\subset M^2$ with $\mu^2(M^2\backslash S) = 0$. Let $z_1\in R_M$. By Proposition~\ref{prop:gammasurjective}, there exists $(x,y)\in\mathcal{R}_M$ such that $z_1=\gamma_M(x,y)$. Write $z=(z_1,z_2,z_3,z_4)=\pi_M(x,y)\in\Lambda$. Since $\pi_M$ is a submersion at $(x,y)$, Lemma~\ref{lemma:regroup} together with (\ref{Risfull}) imply that for $\lambda^{3d}$-a.e. $w\in\Lambda$ in an open neighbourhood $V$ of $z$, \[h(w_1)h(w_2)=h(w_3)h(w_4).\] By shrinking $V$ if necessary, we may assume that $V=(T_1\times T_2\times T_3\times T_4)\cap\Lambda$ where each $T_j$ is an open connected subset of $R_M$, $T_1$ is a ball with center at $z_1$, $T_2$ is a translate of $T_1$ and the projection $\Lambda\to(\reals^d)^2$ given by $(w_1,w_2,w_3,w_4)\mapsto(w_1,w_2)$ maps $V$ onto $T_1\times T_2$. Consider the set $\mathcal{T}$ of all tuples $(w_1,w_2,w_1',w_2',w_3,w_4)$ where $w_1,w_1'\in T_1$, $w_2,w_2'\in T_2$, $w_3\in T_3$, $w_4\in T_4$ and $w_1+w_2=w_1'+w_2'=w_3+w_4$. Then $\mathcal{T}$ is an open subset of a $4d$-dimensional linear space. The projections given by $(w_1,w_2,w_1',w_2',w_3,w_4)\mapsto(w_1,w_2,w_3,w_4)$ and $(w_1,w_2,w_1',w_2',w_3,w_4)\mapsto(w_1',w_2',w_3,w_4)$ are surjective submersions $\mathcal{T}\to V$. Therefore, by Lemma~\ref{lemma:abscont}, for $\lambda^{3d}$-a.e. $(w_1,w_2,w_1',w_2')\in (T_1\times T_2\times T_1\times T_2)\cap\Lambda$, \[h(w_1)h(w_2)=h(w_1')h(w_2').\] By setting $H(u)$ to be the average of $h(w_1)h(w_2)$ over pairs $(w_1,w_2)\in T_1\times T_2$ satisfying $w_1+w_2=u$, we obtain a measurable function $H:T_1+T_2\to\complex$ such that $H(w_1+w_2)=h(w_1)h(w_2)$ for $\lambda^{2d}$-a.e. $(w_1,w_2)\in T_1\times T_2$. Applying Lemma~\ref{lemma:classical}, we deduce that there exist $c\ne 0$ and $v\in\complex^d$ such that $h(w)=c\exp(v\cdot w)$ for $\lambda^d$-a.e. $w\in T_1$. 

Thus we have shown that for each $z\in R_M$, there exist an open connected neighbourhood $B_z\subset R_M$ of $z$, $v_z\in\complex^d$ and $c_z\ne 0$ such that $h(w)=c_z\exp(v_z\cdot w)$ for $\lambda^d$-a.e. $w\in B_z$. Define an equivalence relation $\sim$ on $R_M$ by declaring $w \sim z$ whenever $v_w = v_z$. Observe that if $z_0\in R_M$ then $z \sim z_0$ for any $z\in B_{z_0}$. Therefore, the equivalence relation partitions $R_M$ into a collection of open sets $\{R_j\}_{j\in J}$ such that each $z\in R_j$ has a neighbourhood on which $h$ is of the form $h(z) = c\exp(v_j\cdot z)$, up to a null set, for some $c$ depending on the neighbourhood. 

For $j\in J$, consider the set $M_j$ of all $x\in M$ for which there exists some $y\in M$ such that $x+y\in R_j$. Since $M$ is nowhere-flat, it follows that the sets $\{M_j\}_{j\in J}$ are an open cover for $M$. Furthermore, the sets are disjoint. Indeed, suppose that the open set $M_i\cap M_j$ is non-empty. Then there are $y_1, y_2\in M$ and an open subset $U$ of $M$ such that $U+y_1\in R_i$, $U+y_2\in R_j$, $g(y_1),g(y_2)\ne 0$ and the restricted Cauchy-Pexider equation (\ref{Cauchy}) is satisfied for $S=U_1\times\{y_1,y_2\}$ where $U_1$ is a subset of $U$ with full $\mu$-measure. By shrinking $U$ if necessary, for $\mu$-a.e. $x\in U$, \[g(y_1)^{-1}c\exp(v_i\cdot x)=f(x)=g(y_2)^{-1}c'\exp(v_j\cdot x)\] for some constants $c,c'\ne 0$. By Lemma~\ref{lemma:localise}, this implies that $v_i=v_j$ so $M_i=M_j$.

Since $M$ is connected, it follows that $M_j=M$ for some $j\in J$. This means that each $x\in M$ has an open neighbourhood $U$ in $M$ on which $f$ is of the form $f(x) = c_U\exp(v_j\cdot x)$ except for a $\mu$-null set for some constant $c_U\ne 0$. Appealing to the assumption that $M$ is connected once again, it follows that there is a constant $c_1\ne 0$ such that $f(x) = c_1\exp(v_j\cdot x)$ for $\mu$-almost every all $x\in M$. Similarly, $g(y) = c_2(v_j\cdot y)$ for $\mu$-almost every $y\in M$. 

Thus the proof of Theorem~\ref{theorem:mainmult} will be complete once we establish Proposition~\ref{prop:gammasurjective}.  Lemmata~\ref{lemma:localise}~and~\ref{lemma:regroup} have straighforward analogues for the additive Cauchy-Pexider functional equation with similar proofs; the details are omitted. By combining these with the additive analogue of Lemma~\ref{lemma:classical} (see \cite{christyoungs}) and running through the argument above, we obtain Theorem~\ref{theorem:mainadd}.
\qed

\section{The Proof of Proposition~\ref{prop:gammasurjective}}

\begin{definition}
A \emph{planar circle} in $\mathbb{S}^{d-1}$ is the intersection in $\mathbb{R}^d$ of a two-dimensional linear subspace with $\mathbb{S}^{d-1}$.
\end{definition}

\begin{lemma}\label{lemma:integcurve}
Suppose that $M\subset\reals^d$ is a nowhere-flat hypersurface. Then $M$ is cylinder-free if and only if whenever $U$ is an open subset of $M$, the set $\mathcal{G}(U)\subset\mathbb{S}^{d-1}$ is not contained in a planar circle.
\end{lemma}

\begin{proof}
Suppose that $\mathcal{G}(U)$ is contained in the planar circle $\Pi\cap\mathbb{S}^{d-1}$ where $\Pi\subset\reals^d$ is a two-dimensional plane. Since $M$ is nowhere-flat, for each $x\in U$ there is exactly one non-zero principal curvature with principal direction $p_x$, say, and $\Pi=\text{span}\{\mathcal{G}(x),p_x\}$. Each direction $v\in\Pi^\perp$ then defines a unit vector field of constant direction $v$ on $U$. Thus there is an integral curve which is a line segment in the direction $v$ through each $x\in U$. Choose $x_0\in U$ and consider the smooth curve $\Gamma=(x_0+\Pi)\cap U$ through $x_0$. There is a line segment in each direction $v\in\Pi^\perp$ at each point $x\in\Gamma$ which lies entirely in $U$ and moreover these line segments can be chosen to vary smoothly with $x$ and $v$. Thus, $U$ contains a cylinder. 

The converse follows from the observation that $\mathcal{G}(U)$ is contained in a planar circle whenever $U$ is a cylinder.
\end{proof}

Let $M$ be a connected, nowhere-flat, cylinder-free hypersurface. For each $w\in R_M$, the submanifolds $M$ and $w-M$ of $\mathbb{R}^d$ intersect transversally at a point $x$ if and only if $\overline{\mathcal{G}}(x)\ne\overline{\mathcal{G}}(w-x)$. Write $M_w\subset M\cap(w-M)$ for the set of points where $M$ meets $w-M$ transversally. Then $M_w$ is a non-empty submanifold of $\mathbb{R}^d$ of codimension $2$. Define $D_M=\{w\in R_M\,\, |\,\, \mathcal{G}(M_w) \text{ is not contained in a planar circle} \}$. 

\begin{lemma}\label{lemma:Ddense}
Let $M$ be nowhere-flat and cylinder-free. Then $D_M$ is open and dense in $R_M$.
\end{lemma}

\begin{proof}
Let $w\in D_M$. Then there exist $x_1, x_2, x_3\in M_w$ such that $\mathcal{G}(x_1), \mathcal{G}(x_2), \mathcal{G}(x_3)$ are not contained in a planar circle. Recall the surjective submersion $\alpha_M:\mathcal{M}\to R_M$, $(x,y)\mapsto x+y$. Let $\theta_j:U\to \mathcal{M}$ for $j=1, 2, 3$ be smooth local sections of $\alpha_M$ defined on some open $U\subset R_M$ containing $w$ such that $\theta_j(w)=(x_j,w-x_j)$. Let $p:\mathcal{M}\to M$ be the restriction to $\mathcal{M}$ of the projection onto the first coordinate. For all $z$ in a small neighbourhood of $w$ in $U$, the points $\mathcal{G}(p\circ\theta_1(z)), \mathcal{G}(p\circ\theta_2(z)), \mathcal{G}(p\circ\theta_3(z))$ will also not be contained in a planar circle. Thus $D_M$ is open.

Suppose for contradiction that $D_M$ is not dense. Let $S\subset R_M$ be non-empty and open such that for all $w\in S$, the set $\mathcal{G}(M_w)$ is contained in the planar circle $\Pi_w\cap \mathbb{S}^{d-1}$ where $\Pi_w$ is a two-dimensional plane. For any $x\in M_w$, $\overline{\mathcal{G}}(x)\ne\overline{\mathcal{G}}(w-x)$ so $\Pi_w = \text{ span}\{\mathcal{G}(x),\mathcal{G}(w-x)\}$ and the tangent space
\[T_xM_{w}=T_xM\cap T_{w-x}M=\text{span}\{\mathcal{G}(x),\mathcal{G}(w-x)\}^\perp=\Pi_{w}^\perp.\]

Fix $w_0\in S$ and $x_0\in M_{w_0}$. Set $y_0=w_0-x_0\in M$. In particular, this means that $\overline{\mathcal{G}}(y_0)=\overline{\mathcal{G}}(w_0-x_0)\ne\overline{\mathcal{G}}(x_0)$ and there exists a small open neighbourhood $B$ of $y_0$ in $M$, such that for all $y\in B$, $\overline{\mathcal{G}}(y)\ne\overline{\mathcal{G}}(x_0)$  and $w(y)=x_0+y\in S$. Since $\mathcal{G}(M_{w(y)})\subset \Pi_{w(y)}$, it follows that  $d\mathcal{G}_{x_0}(T_{x_0}M_{w(y)})\subset \Pi_{w(y)}$. Thus for any unit vector $v\in\Pi_{w(y)}^\perp=\text{span}\{\mathcal{G}(x_0),\mathcal{G}(y)\}^\perp$, 
\begin{equation}\label{fundform}
d\mathcal{G}_{x_0}(v)\cdot v = 0.\end{equation}

Let $\{q_1,\ldots, q_{d-1}\}$ be a set of principal directions at $x_0$ with corresponding principal curvatures $\kappa_1, \ldots, \kappa_{d-1}$. So $\{q_1,\ldots, q_{d-1}\}$ is a basis for $T_{x_0}M$ and $\{q_1,\ldots, q_{d-1}, \mathcal{G}(x_0)\}$ is an orthonormal basis for $\mathbb{R}^d$. Furthermore, the set $\{q_1,\ldots, q_{d-1}\}$ is a basis of eigenvectors for the linear map $d\mathcal{G}_{x_0}$ with eigenvalues $-\kappa_1, \ldots, -\kappa_{d-1}$.

Define $\overline{\bf r}(y)=(q_1\cdot \mathcal{G}(y),\ldots,q_{d-1}\cdot \mathcal{G}(y))$, ${\bf v}=(v_1,\ldots, v_{d-1})$ and ${\bf k}=(\kappa_1,\ldots,\kappa_{d-1})$ viewed as elements of $\mathbb{R}^{d-1}$. Since $\overline{\mathcal{G}}(y)\ne\overline{\mathcal{G}}(x_0)$, it follows that $\overline{\bf r}(y)\ne 0$ and we can define ${\bf r}(y)=\overline{\bf r}(y)/\|\overline{\bf r}(y)\|$. Also define \[H_{\bf k}=\{{\bf u}\in\mathbb{R}^{d-1}\,|\,\sum_{j=1}^{d-1}\kappa_ju_j^2=0\}.\] Then, equation (\ref{fundform}) is equivalent to the assertion that whenever ${\bf v}\in\mathbb{S}^{d-2}$ and ${\bf v}\cdot {\bf r}(y)=0$ it follows that ${\bf v}\in H_{\bf k}$. Since $M$ is nowhere-flat, ${\bf k}\ne 0$. Suppose without loss of generality that $\kappa_{d-1}\ne 0$. Let $\Phi=\{{\bf u}\in\mathbb{R}^{d-1}\,|\,u_{d-1}\ne 0\}$. Then the set $H_{\bf k}\cap\mathbb{S}^{d-2}\cap\Phi$ is a $(d-3)$-dimensional submanifold of $\mathbb{R}^{d-1}$. Observe that ${\bf r}(y)$ varies smoothly with $y\in B$ on $\mathbb{S}^{d-2}$. If it is not constant, then \[\bigcup_{y\in B}\text{span}\{{\bf r}(y)\}^\perp\cap\mathbb{S}^{d-2}\] contains an open subset of $\mathbb{S}^{d-2}$. For every $y\in B$, the $(d-3)$-dimensional submanifold $\text{span}\{{\bf r}(y)\}^\perp\cap\mathbb{S}^{d-2}$ is a subset of $H_{\bf k}\cap\mathbb{S}^{d-2}$, thus $H_{\bf k}\cap\mathbb{S}^{d-2}$ contains an open subset of $\mathbb{S}^{d-2}$. Since $\Phi\cap\mathbb{S}^{d-2}$ is a dense open subset of $\mathbb{S}^{d-2}$, this implies that an open subset of a $(d-2)$-dimensional manifold is contained in a $(d-3)$-dimensional manifold; this is a contradiction. Therefore ${\bf r}(y)=(r_1,\ldots,r_{d-1})$ is constant.

Thus for all $y\in B$, \[\mathcal{G}(y)\in\text{span}\{\mathcal{G}(x_0),\sum_{j=1}^{d-1}r_jq_j\}.\] By Lemma~\ref{lemma:integcurve}, this contradicts $M$ being cylinder-free.
\end{proof}

We now prove Proposition~\ref{prop:gammasurjective}. Let $w\in R_M$. By Lemma~\ref{lemma:geomcond}, it suffices to show that there exists $(x,y)\in\mathcal{P}_M$ satisfying (\ref{smooth}), (\ref{regular1}) and (\ref{regular2}).

The set $M+M_w$ contains a subset which is open in $\reals^d$. Indeed, consider the smooth map $\tau:M\times M_w\to\mathbb{R}^d$, $(x,s)\mapsto x+s$. Then $\tau$ is a submersion at $(x,s)$ if and only if $\mathcal{G}(x)$ does not lie in the span of $\mathcal{G}(s)$ and $\mathcal{G}(w-s)$. Since $M$ is cylinder-free the set $\mathcal{G}(M)$ does not lie in any planar circle, thus given any $s\in M_w$ there exists $x\in M$ such that $(x,s)$ is a regular point of $\tau$. Therefore the image of $\tau$ must contain an open set.

By Lemma~\ref{lemma:Ddense}, there thus exist $x_1,y_1,y_2\in M$ such that $x_1+y_2\in D_M$, $x_1+y_1=w$ and $\overline{\mathcal{G}}(x_1)\ne\overline{\mathcal{G}}(y_1)$.Since $M$ is nowhere-flat, $y_2$ may be perturbed in a small $M$-open set if necessary so that $\overline{\mathcal{G}}(y_2)\ne\overline{\mathcal{G}}(x_1)$.

Since $M$ is cylinder-free, there exists $x_2\in M$ such that 
\begin{equation}\label{equ1}\mathcal{G}(x_2)\notin\text{ span}\{\mathcal{G}(x_1), \mathcal{G}(y_1)\}.
\end{equation} By perturbing $x_2$ if necessary we may assume in addition that 
\begin{equation}\label{equ2}\overline{\mathcal{G}}(x_2)\ne\overline{\mathcal{G}}(y_2).
\end{equation}

The points $y_3$ and $x_4$ will be chosen to be close to $y_1$ and $x_2$ respectively. To this end, set $y_3=y_1$ and $x_4=x_1$ for the moment. Since $x_1+y_2\in D_M$, there exists $y_4$ such that $y_4\in M\cap(x_1+y_2-M)$ and
 \begin{equation}\label{equ3}\mathcal{G}(y_4)\notin\text{ span}\{\mathcal{G}(x_2), \mathcal{G}(y_2)\}.
\end{equation} Let $x_3=x_1+y_2-y_4\in M$.

Define $\Delta = \Delta(x_1,x_2,y_1,y_2) = \text{span}\{\mathcal{G}(x_1),\mathcal{G}(y_1)\}\cap\text{ span}\{\mathcal{G
}(x_2),\mathcal{G}(y_2)\}$. The choice of $x_1, y_1, x_2, y_2$ implies that $\mathcal{G}(x_2)\notin \Delta$ and $\dim\Delta\le 1$. Then, by (\ref{equ1}), (\ref{equ2}), (\ref{equ3}), it follows that 
\begin{equation}\label{equ4}\mathcal{G}(x_4)\notin\text{ span}\{\mathcal{G}(y_4)\}+\Delta
\end{equation} 
so $\overline{\mathcal{G}}(x_4)\ne\overline{\mathcal{G}}(y_4)$ and $\Delta\cap\text{span}\{\mathcal{G}(x_4),\mathcal{G}(y_4)\}=\{0\}$. Therefore (\ref{smooth}) and (\ref{regular2}) are satisfied. Moreover, (\ref{regular1}) is also satisfied except perhaps for the statement $\overline{\mathcal{G}}(x_3)\ne\overline{\mathcal{G}}(y_3)$.

Suppose that $\overline{\mathcal{G}}(x_3)=\overline{\mathcal{G}}(y_3)$ for this choice of $(x,y)$. Keep $x_1$, $x_3$, $y_1$, $y_2$, $y_3$ and $y_4$ fixed. Then $x_2$ may be perturbed freely in an open subset $B$ of $M$, with $x_4$ set to $x_2$ for each choice, while preserving the conditions (\ref{equ1}), (\ref{equ2}) and (\ref{equ3}). Furthermore, for each choice of $x_2\in B$, we may perturb $y_3$ in a small open subset of $M_{y_1+x_2}$ and redefine $x_4=y_1+x_2-y_3$ while still preserving condition (\ref{equ4}). The next lemma completes the proof of Proposition~\ref{prop:gammasurjective}.

\begin{lemma}
Let $B\subset M$ be non-empty and open. Suppose that $y\in M$ satisfies $\overline{\mathcal{G}}(y)\notin\overline{\mathcal{G}}(B)$. Then, for any open neighbourhood $U\subset M$ of $y$, there exists $x\in B$ and $y'\in M_{x+y}\cap U$ such that $\overline{\mathcal{G}}(y')\ne\overline{\mathcal{G}}(y)$.
\end{lemma}

\begin{proof}
Suppose not. Then for each $x\in B$, the Gauss normal map $\mathcal{G}$ is constant along $M_{x+y}$ near $y$. Thus, the $(d-2)$-dimensional tangent plane $T_yM_{x+y}$ must lie entirely in the span of the principal directions at $y\in M$ corresponding to zero principal curvatures. Since $M$ is nowhere-flat, there exists a principal direction $p$ at $y$ with non-zero principal curvature. Since $T_yM_{x+y}=T_yM\cap T_y(x+y-M)=T_yM\cap T_xM$, this means that $p$ lies in the span of $\mathcal{G}(y)$ and $\mathcal{G}(x)$. Since this is true for all $x\in B$ and $\mathcal{G}(y)\cdot p=0$ it follows that $\mathcal{G}(B)$ lies entirely in the span of $p$ and $\mathcal{G}(y)$. By Lemma~\ref{lemma:integcurve}, this contradicts $M$ being cylinder-free.
\end{proof}

\section{Nearly-Multiplicative Functions on Hypersurfaces}\label{section:app}

In this section we establish Theorem~\ref{theorem:appmain}, which is an approximate version of the main result Theorem~\ref{theorem:mainmult}. By replacing multiplication with addition where appropriate, the same argument also establishes the approximate version of Theorem~\ref{theorem:mainadd}.

If $F:(0,\delta_0)\to\complex$ is a function with domain the open interval $(0,\delta_0)$ for some $\delta_0>0$ and $\tau_1,\ldots,\tau_k$ is a list of parameters, we will use the Landau notation \[F(\delta)=o_{\tau_1,\ldots,\tau_k}(1)\] to mean that $F(\delta)\to 0$ as $\delta\to 0$ in a way which only depends on $\tau_1,\ldots,\tau_k,M$---that is, there exists some $0<\delta_1<\delta_0$ and a function $\nu:(0,\delta_1)\to(0,\infty)$ depending only on $\tau_1,\ldots,\tau_k$ and $M$ satisfying $\nu(\delta)\to 0$ as $\delta\to 0$ such that for all $0<\delta<\delta_1$ \[|F(\delta)|\le\nu(\delta).\] We will write $F(\delta)=G(\delta)+o_{\tau_1,\ldots,\tau_k}(1)$ to mean $F(\delta)-G(\delta)=o_{\tau_1,\ldots,\tau_k}(1)$.

Let $M$ be as in the statement of Theorem~\ref{theorem:appmain}. Suppose that $f,g:M\to\complex$ are measurable functions vanishing only on a $\mu$-null set and $h:2M\to\complex$ is measurable. For each $\delta>0$ define \[E_\delta=\{(x,y)\in M^2\,\,|\,\,|f(x)g(y)h(x+y)^{-1}-1|>\delta\}.\] By closely following the discussion in \S\ref{section:strategy} which relies on Proposition~\ref{prop:gammasurjective}, replacing null sets with small sets where appropriate, we obtain the following proposition; the precise details are omitted.

\begin{proposition}
For all $z\in R_M$ there exist a non-empty open ball $T_z\subset\reals^d$ with center $z$ and a translate $B_z$ of $T_z$ with the following property. If $\mu^2(E_\delta)<\delta$ then there exist a measurable function $H_z^\delta:T_z + B_z\to\complex$ and a set $\mathcal{B}_z^\delta\subset T_z\times B_z$ such that 
\[
h(w_1)h(w_2)H_z^\delta(w_1+w_2)^{-1} = 1+o(1)
\]
for all $(w_1,w_2)\in\mathcal{B}_z^\delta$ and \[\lambda^{2d}(\mathcal{B}_z^\delta)=o_z(1).\]
\end{proposition}

By Lemma~\ref{lemma:classical}, the proposition implies that for each $z\in R_M$ if $\mu^2(E_\delta)<\delta$ then there exist $c_z^\delta\in\complex\backslash\{0\}$ and $v_z^\delta\in\complex^d$ satisfying \[h(w)(c_z^\delta\exp(v_z^\delta\cdot w))^{-1}=1+o_z(1)\] for all $w\in T_z$ except for a set of $\lambda^d$-measure $o_z(1)$.  Define \[\mathcal{N}_z=\{x'\in M\,\,|\,\,\text{there exists } y\in M \text{ such that }\overline{\mathcal{G}}(x')\ne\overline{\mathcal{G}}(y)\text{ and }x'+y\in T_z\}.\] Applying (\ref{convequalsleb}) we deduce that whenever $\mu^2(E_\delta)<\delta$ there exist $b_z^\delta\in\complex\backslash\{0\}$ and $v_z^\delta\in\complex^d$ such that 
\begin{equation}\label{fisaffine}
f(x)(b_z^\delta\exp(v_z^\delta\cdot x))^{-1}=1+o_{z}(1)
\end{equation} 
for all $x\in \mathcal{N}_z$ except for a set of $\mu$-measure $o_z(1)$. 

The next lemma may be interpreted as an approximate analogue of Lemma~\ref{lemma:localise}.

\begin{lemma}\label{lemma:appaffine}
Let $M$ be a nowhere-flat bounded hypersurface with finite induced measure $\mu$. Let $c\in\complex\backslash\{0\}$ and $v\in\complex^d$. Given $\epsilon>0$ there exists $\delta=\delta(\epsilon,M)>0$ such that if \[\mu(\{|c\exp(v\cdot x)-1|>\delta\})<\delta\] then $|c-1|$, $\norm{v}$ are less than $\epsilon$.  
\end{lemma}

\begin{proof}
Let $L_\delta=\{x\in M\,\,|\,\,|c\exp(v\cdot x) - 1|\le \delta\}$ and suppose that $\mu(M\backslash L_\delta)<\delta<1$. For all $(x,y)\in L_\delta^2$ it follows that $|c^2\exp(v\cdot(x+y))-1|\le 3\delta$. Since $M$ is nowhere-flat, $\mu^2(\{(x,y)\in M^2\,\,|\,\, \overline{\mathcal{G}}(x)=\overline{\mathcal{G}}(y)\})=0$ so 
\[\mu^2(\{(x,y)\in M^2\,\,|\,\,\overline{\mathcal{G}}(x)\ne\overline{\mathcal{G}}(y)\text{ and } x+y\in 2L_\delta\})\ge\mu^2(L_\delta^2).\]

By Lemma~\ref{lemma:abscont}, there exists a function $l:(0,1)\to(0,\infty)$ independent of $c$, $v$ and $\delta$ such that $l(\delta)\to 0$ as $\delta\to 0$ and $\lambda^d(R_M\backslash R_M\cap 2L_\delta)<l(\delta)$. Thus for all $w\in R_M$ outside a set of $\lambda^d$-measure $l(\delta)$, \[|c^2\exp(v\cdot w)-1|\le 3\delta.\]
Since $R_M$ is open and non-empty, it follows that there exists a function $\epsilon:(0,1)\to(0,\infty)$ independent of $c$, $v$ and $\delta$ such that $\epsilon(\delta)\to 0$ as $\delta\to 0$ and $|c-1|, \norm{v}\le \epsilon(\delta)$.
\end{proof}

Let $\epsilon>0$ be fixed. Let $\tau>0$ be a small positive parameter to be chosen below. Then there exists a set $\{z_1,\ldots,z_{n}\}$ of elements of $M$ depending on $\tau$ such that $\mathcal{N}=\cup_{j=1}^{n}\mathcal{N}_{z_j}$ is connected and  \[\mu(M\backslash\mathcal{N})<\tau.\] By (\ref{fisaffine}) and Lemma~\ref{lemma:appaffine}, for small $\delta>0$ if $\mu^2(E_\delta)<\delta$ then for each $1\le i,j\le n$, either the sets $\mathcal{N}_{z_i}$ and $\mathcal{N}_{z_j}$ are disjoint or 
\[|b_{z_i}^\delta (b_{z_j}^\delta)^{-1}-1|,\norm{v_{z_i}^\delta-v_{z_j}^\delta}=o_{\tau}(1).\] Set $b_\delta=b_{z_1}^\delta$ and $v_\delta=v_{z_1}^\delta$. Since $\mathcal{N}$ is connected, it follows that for each $1\le j\le n$,
\begin{equation}\label{affinity}
|b_\delta (b_{z_j}^\delta)^{-1}-1|,\norm{v_\delta-v_{z_j}^\delta}= o_\tau(1).
\end{equation}
Observe that (\ref{affinity}) implies that for all $x\in M$ \[(b_\delta\exp(v_\delta\cdot x))(b_{z_j}^\delta\exp(v_{z_j}^\delta\cdot x))^{-1}=1+o_\tau(1).\] Therefore (\ref{fisaffine}) implies that if $\mu(E_\delta)<\delta$ then there exist $b_\delta\in\complex\backslash \{0\}$, $v_\delta\in\complex^d$ and a set $\mathcal{N}^\delta\subset \mathcal{N}$ depending on $\tau$ such that $\mu(\mathcal{N}\backslash \mathcal{N}^\delta)=o_\tau(1)$ and \[f(x)(b_\delta\exp(v_\delta\cdot x))^{-1}=1+o_\tau(1)\] for all $x\in \mathcal{N}^\delta$. Choosing $\tau<\epsilon/2$, say, and taking $\delta$ sufficiently small completes the proof of Theorem~\ref{theorem:appmain}.
\qed

\section{Multiplicative Functions on Curves}\label{section:curves}

Let $I\subset\reals$ be an open interval and suppose that $\gamma:I\to\reals^d$ is a smooth isometric---that is, unit speed---embedding. Suppose that no open subset of $\Gamma=\gamma(I)$ lies in an affine hyperplane. This is equivalent to demanding that whenever $U\subset I$ is open, there exist $u_1,\ldots,u_d\in U$ such that the set of vectors $\{\dot{\gamma}(u_1),\ldots,\dot{\gamma}(u_d)\}$ is a basis for $\reals^d$.  Furthermore, arguing similarly to Lemma~\ref{lemma:localise}, this means that we may assume without loss of generality that $I$ is bounded.

In this section we prove Theorem~\ref{theorem:curves}; the additive analogue follows similarly. For each $0\le j\le d$, let $f_j:I\to \complex$ be a measurable function which only vanishes on a $\lambda$-null set. Assume, in addition, that each $f_j$ is locally $\lambda$-bounded in the sense that $\text{ess sup}_{x\in K} |f_j(x)|$ is finite for each compact $K\subset I$. Suppose that there exists a measurable function $F:(d+1)\Gamma\to\complex$ such that
\begin{equation}\label{dfoldCauchy}
\prod_{j=0}^df_j(x_j)=F\big(\sum_{j=0}^d\gamma(x_j)\big)\text{ for }\lambda^{d+1}\text{-a.e. }(x_0,\ldots,x_d)\in I^{d+1}.
\end{equation}

Define \[\mathcal{A}=\{(x_0,\ldots,x_d)\in I^{d+1}\,\,|\,\,\text{span}\{\dot{\gamma}(x_0)),\ldots,\dot{\gamma}(x_d)\}=\reals^d\}.\]
The smooth map $I^{d+1}\to\reals^d$ given by $(x_0,\ldots,x_d)\mapsto\sum_{j=0}^d\gamma(x_j)$ restricts to a submersion $\alpha:\mathcal{A}\to\reals^d$. By Lemma~\ref{lemma:abscont} and (\ref{dfoldCauchy}), the assumed local boundedness of $f_j$ for $0\le j\le d$ implies that $F$ is locally $\lambda^d$-integrable on the open set $\alpha(\mathcal{A})\subset\reals^d$. Choose $u_1,\ldots,u_d\in I$ such that $\{\dot{\gamma}(u_1),\ldots,\dot{\gamma}(u_d)\}$ is a basis for $\reals^d$. By the inverse function theorem, the smooth map $\beta:I^d\to\reals^d$ given by $(x_1,\ldots,x_d)\mapsto\sum_{j=1}^d\gamma(x_j)$ is a diffeomorphism from a bounded open neighbourhood $U\subset  I^d$ of $(u_1,\ldots,u_d)$ onto a non-empty open ball $V=B_r(z)\subset\reals^d$ and $I\times U\subset\mathcal{A}$. By shrinking $U$ and $V$ if necessary, we may assume that, in addition, \[\chi=\int_U\big(\prod_{j=1}^df_j(x_j)\big)dx_1\ldots dx_d\ne 0.\] Therefore $f_0$ agrees $\lambda$-almost everywhere with the continuous function \begin{equation}\label{eqncont}
g_0:x_0\mapsto\frac{\int_UF\big(\gamma(x_0)+\sum_{j=1}^d\gamma(x_j)\big)dx_1\ldots dx_d}{\int_U\big(\prod_{j=1}^df_j(x_j)\big)dx_1\ldots dx_d}
\end{equation} and similarly, for each $1\le j\le d$, the function $f_j$ agrees $\lambda$-almost everywhere with a continuous function $g_j$. By Lemma~\ref{lemma:abscont} and (\ref{dfoldCauchy}), $F|_\mathcal{\alpha(A)}$ agrees almost everywhere with a continuous function $G$.

Thus,
\[g_0(x)=\chi^{-1}\int_{V}G(\gamma(x)+u)J_\beta(u)du=\chi^{-1}\int_{V+\gamma(x)}G(u)J_\beta(u-\gamma(x))du\]
where $J_\beta$ denotes the Jacobian of $\beta$. By the continuity of $G$ and the smoothness of $\gamma$ and $J_\beta$, it follows by the Leibniz integral theorem that $g_0$ is differentiable. Similarly, for each $1\le j\le d$, the function $g_j$ is differentiable. Therefore $G$ is differentiable and 
\begin{equation}
\prod_{j=0}^dg_j(x_j)=G\big(\sum_{j=0}^d\gamma(x_j)\big)\text{ for all }(x_0,\ldots,x_d)\in \mathcal{A}.
\end{equation}
Differentiating with respect to $x_0$,
\begin{equation}\label{difffunc}
g_0'(x_0)\prod_{j=1}^dg_j(x_j)=\dot{\gamma}(x_0)\cdot\nabla G\big(\sum_{j=0}^d\gamma(x_j)\big)\text{ for all }(x_0,\ldots,x_d)\in \mathcal{A}.
\end{equation}

Fix $x_0\in I$. Let $V=B_r(z)$ be the open ball defined above. Choose $v_1,\ldots,v_d\in \gamma^{-1}(B_r(\gamma(x_0)))$ such that $\dot{\gamma}(v_1),\ldots,\dot{\gamma}(v_d)$ span $\reals^d$. Without loss of generality we may assume that $g_0(v_j)\ne 0$ for each $1\le j\le d$. Let $S$ be a non-empty connected open neighbourhood of $\gamma(x_0)+z$ contained in $\cap_{j=1}^d(\gamma(v_j)+B_r(z))$. By equation (\ref{difffunc}), for each $1\le j\le d$,
\begin{equation}
\partial_{\dot{\gamma}(v_j)}G(s)=\frac{g_0'(v_j)}{g_0(v_j)}G(s) \text{ for all }s\in S
\end{equation}
where $\partial_wG$ denotes the directional derivative of $G$ in the direction $w$. Since the directions $\dot{\gamma}(v_1),\ldots,\dot{\gamma}(v_d)$ span $\reals^d$, it follows that there exist $c_S\in\complex$ and $\xi_S\in\complex^d$ such that $G(s) = c_S\exp(\xi_S\cdot s)$ for all $s\in S$. Thus there exist a neighbourhood $B\subset I$ of $x_0$, $c_B\ne 0$ and $\xi_B\in\complex^d$ such that for all $x\in B$, $g_0(x)=c_B\exp(\xi_B\cdot\gamma(x))$. Arguing as in the proof of Lemma~\ref{lemma:localise} using the assumption that $\Gamma$ is not contained in an affine hyperplane, this implies that there exist $c_0\ne 0$ and $\xi_0\in\complex^d$ such that $g_0(x)=c_0\exp(\xi_0\cdot\gamma(x))$ for all $x\in I$. Moreover, $c_0$ and $\xi_0$ are uniquely determined.

This proves Theorem~\ref{theorem:curves} in the case when each function $f_j$ is locally bounded. Therefore to prove the theorem in its full generality it suffices to establish that the functional equation (\ref{dfoldCauchy}) forces measurable functions $f_0,\ldots,f_d$ to be locally bounded. This type of strategy is standard; it is discussed in \cite{regularity}, for example.

Suppose that the $f_j$ and $F$ are merely measurable and let $U$, $V$ and the diffeomophism $\beta:U\to V$ be as defined above. By Lusin's theorem applied to the function $(x_1,\ldots,x_d)\mapsto\prod_{j=1}^d|f_j(x_j)|$, there exists a compact set $\mathcal{K}\subset U$ of positive $\lambda^d$-measure such that \[\inf_{(x_1,\ldots,x_d)\in\mathcal{K}}\prod_{j=1}^d|f_j(x_j)|\ge c_0>0\] for some positive constant $c_0$. There exists $I_1\subset I$ such that $\lambda(I\backslash I_1)=0$ and for each $x\in I_1$ the functional equation (\ref{dfoldCauchy}) is satisfied for the point $(x,x_1,\ldots,x_d)$ for $\lambda^d$-almost every $(x_1,\ldots,x_d)\in \mathcal{K}$. 

By Lusin's theorem applied to $F|_{\gamma(I)+V}$, there exist a compact subset $T$ of the open set $\gamma(I)+V$ and a constant $C<\infty$ such that 
\begin{equation}\label{lusinapp}\lambda^{d}((\gamma(I)+V)\backslash T)<\lambda^{d}(\beta(\mathcal{K}))\ne 0,\end{equation}
and for all $w\in T$, $|F(w)|\le C$. Let $x\in I_1$. Then $\lambda^d(\gamma(x)+\beta(\mathcal{K}))=\lambda^{d}(\beta(\mathcal{K}))$ so (\ref{lusinapp}) implies that $T\cap(\gamma(x)+\beta(\mathcal{K}))$ has positive $\lambda^d$-measure. Thus, there exists $w\in\mathcal{K}$ such that $|F(\gamma(x)+\beta(w))|\le C$. Hence $|f_0(x)|\le c_0^{-1}C$ and $f_0$ is $\lambda$-bounded.
\qed

\section{Three-fold Multiplicative Functions on Hypersurfaces}\label{section:threefold}

In this section we prove Theorem~\ref{theorem:threefold} by appropriately adapting the setup in \S\ref{section:strategy} for suitably non-degenerate hypersurfaces and appealing to Theorem~\ref{theorem:curves} with $d=2$ for the remaining cases.

Let $d\ge 3$. Let $M\subset\reals^d$ be a nowhere-flat hypersurface. By Lemma~\ref{lemma:localise}, we can assume without loss of generality that $M$ is bounded and has finite induced measure $\mu$; this lemma also implies the stated uniqueness.

Write points in $(\reals^d)^6\times(\reals^d)^6\times(\reals^d)^6$ as $(x,y,z)$ where $x=(x_1,\ldots,x_6)$, $y=(y_1,\ldots,y_6)$ and $z=(z_1,\ldots,z_6)$ with $x_j, y_j,z_j\in\reals^d$. Let $\Pi$ be the hyperplane in $(\reals^d)^6\times(\reals^d)^6\times(\reals^d)^6$ defined by the three equations \[x_1+y_2+z_3=x_4+y_5+z_6,\] \[x_2+y_3+z_1=x_5+y_6+z_4,\]\[x_3+y_1+z_2=x_6+y_4+z_5.\]  Let $\mathcal{P}_M=(M^6\times M^6\times M^6)\cap\Pi$. Write $\mathcal{S}_M$ for the set of smooth points of $\mathcal{P}_M$. Then $\mathcal{S}_M$ is a submanifold of $\mathbb{R}^{18d}$ of dimension $18(d-1)+15d-18d = 15d-18$. Write $\sigma$ for the volume measure on $\mathcal{S}_M$ associated to the induced Riemannian structure. Observe that the point $(x,y,z)\in\mathcal{P}_M$ lies in $\mathcal{S}_M$ if and only if 
\begin{eqnarray}\label{smooth3}
&T_{x_1}M+T_{y_2}M+T_{z_3}M+T_{x_4}M+T_{y_5}M+T_{z_6}M=\reals^d,& \\\label{smooth4}
&T_{x_2}M+T_{y_3}M+T_{z_1}M+T_{x_5}M+T_{y_6}M+T_{z_4}M=\reals^d,& \\\label{smooth5}
&T_{x_3}M+T_{y_1}M+T_{z_2}M+T_{x_6}M+T_{y_4}M+T_{z_5}M=\reals^d.&
\end{eqnarray}

Consider the $5d$-dimensional hyperplane $\Lambda\subset(\mathbb{R}^d)^6$ defined by $w_1+w_2+w_3=w_4+w_5+w_6$. The linear addition map $(\mathbb{R}^d)^6\times(\mathbb{R}^d)^6\times(\reals^d)^6\to(\mathbb{R}^d)^6$, $(x,y,z)\mapsto x+y+z$ restricts to a smooth map $\pi_M:\mathcal{S}_M\to\Lambda$. Write $\mathcal{R}_M$ for the set of regular points of $\pi_M$. Note that a necessary condition for $(x,y,z)\in\mathcal{S}_M$ to lie in $\mathcal{R}_M$ is that
\begin{equation}\label{reg3}
T_{x_j}M+T_{y_j}M+T_{z_j}M=\reals^d\text{ for all }1\le j\le 6.
\end{equation}

 For $S\subset M^3$, define $\mathcal{S}_M(S) = \{(x,y,z)\in\mathcal{S}_M\,|\, (x_i,y_j,z_k)\in S\,\, \text{for all} \,\,1\le i,j,k\le 6\}$. Let $f_1,f_2,f_3:M\to\complex$, and $F:3M\to\complex$ be measurable functions. Suppose that 
\begin{equation}\label{threeCP}
f_1(x)f_2(y)f_3(z)=F(x+y+z)\text{ for all }(x,y,z)\in S\subset M^3.
\end{equation} Then, similarly to the proof of Lemma~\ref{lemma:regroup}, whenever $w\in\Lambda$ is in the image of $\mathcal{S}_M(S)$ under $\pi_M$,  \[F(z_1)F(z_2)F(z_3)=F(z_4)F(z_5)F(z_6).\] 

The addition map $M^3\to\reals^d$ restricts to a smooth submersion $\alpha_M:\mathcal{M}\to R_M$ where we define $\mathcal{M}=\{(x,y,z)\in M^3\,\,|\,\,T_xM+T_yM+T_zM=\reals^d\}$, the set of regular points of the addition map, and $R_M=\{x+y+z\,\,|\,\,x,y,z\in M\text{ and }T_xM+T_yM+T_zM=\reals^d\}$. Then $R_M$ is an open subset of $\reals^d$ which is dense in $3M$. By (\ref{reg3}), we can also define the submersion $\gamma_M:\mathcal{R}_M\to R_M$ given by $(x,y,z)\mapsto x_1+y_1+z_1$. As before, a key ingredient will be the surjectivity of this submersion; we will assume that $\mathcal{G}(M)\subset\mathbb{S}^{d-1}$ is not contained in a planar circle so that $M$ is not itself a cylinder. For the case when $M$ is a cylinder, we will instead apply Theorem~\ref{theorem:curves}.

\begin{proposition}\label{prop:gamma3surj}
Let $M\subset\reals^d$ be a nowhere-flat hypersurface such that the set $\mathcal{G}(M)\subset\mathbb{S}^{d-1}$ is not contained in a planar circle. Then $\gamma_M:\mathcal{R}_M\to R_M$ is surjective.
\end{proposition}

\begin{proof}
Let $w\in R_M$. Write $w=x_1+y_1+z_1$ for $x_1,y_1,z_1\in M$ such that $T_{x_1}M+T_{y_1}M+T_{z_1}M=\reals^d$. By relabelling if necessary we may assume that $\overline{\mathcal{G}}(x_1)\ne\overline{\mathcal{G}}(y_1), \overline{\mathcal{G}}(z_1)$. Set $x_4=x_1$, $y_4=y_1$, $z_4=z_1$. 

Since $\mathcal{G}(M)$ is not contained in a planar circle, there exist $x_2,y_2\in M$ such that $\mathcal{G}(x_2),\mathcal{G}(y_2)\notin\text{span}\{\mathcal{G}(x_1),\mathcal{G}(z_1)\}$. Since $M$ is nowhere-flat, we may take $\overline{\mathcal{G}}(x_2)\ne\overline{\mathcal{G}}(y_2)$. Define $\Delta=\text{span}\{\mathcal{G}(x_1),\mathcal{G}(z_1)\}\cap\text{span}\{\mathcal{G}(x_2),\mathcal{G}(y_2)\}$. Then $\dim\Delta\le 1$. Set $z_2=x_1$, $x_5=x_2$, $y_5=y_2$, $z_5=z_2$. 

Let $x_3\in M$ be any element of $M$. Choose $y_3\in M$ such that $\mathcal{G}(y_3)\notin\Delta$. Since $\mathcal{G}(M)$ is not contained in a planar circle, there exists $z_3\in M$ such that $\mathcal{G}(z_3)\notin \text{span}\{\mathcal{G}(y_3)\}+\Delta$. Then \begin{equation}\label{intersections} \text{span}\{\mathcal{G}(x_1),\mathcal{G}(z_1)\}\cap\text{span}\{\mathcal{G}(x_2),\mathcal{G}(y_2)\}\cap\text{span}\{\mathcal{G}(y_3),\mathcal{G}(z_3)\}=\{0\}.\end{equation} Set $x_6=x_3$, $y_6=y_3$, $z_6=z_3$.

Then \[x_1+y_2+z_3=x_4+y_5+z_6,\] \[x_2+y_3+z_1=x_5+y_6+z_4,\] \[x_3+y_1+z_2=x_6+y_4+z_5\] so the point $(x,y,z)$ for $x=(x_1,\ldots,x_6)$, $y=(y_1,\ldots,y_6)$, $z=(z_1,\ldots,z_6)$ lies in $\mathcal{P}_M$. By construction, $(x,y,z)$ satisfies (\ref{smooth3}), (\ref{smooth4}) and (\ref{smooth5}) so $(x,y,z)\in\mathcal{S}_M$.

We will identify $T_{(x,y,z)}\mathcal{S}_M$ with $(\prod_{j=1}^6T_{x_j}M\times\prod_{j=1}^6T_{y_j}M\times\prod_{j=1}^6T_{z_j}M)\cap\Pi$ and $T_{x+y+z}\Lambda$ with $\Lambda$ in the canonical way. Then $(d\pi_M)_{(x,y,z)}:T_{(x,y,z)}\mathcal{S}_M\to T_{x+y+z}\Lambda$ is given by \[(d\pi_M)_{(x,y,z)}(u,v,w)=u+v+w\] where the tuples $u=(u_1,\ldots,u_6)$, $v=(v_1,\ldots,v_6)$, $w=(w_1,\ldots,w_6)$ satisfy $u_j\in T_{x_j}M$, $v_j\in T_{y_j}M$, $w_j\in T_{z_j}M$ for $1\le j\le 6$ and the linear equations \[u_1+v_2+w_3=u_4+v_5+w_6,\]  \[u_2+v_3+w_1=u_5+v_6+w_4,\] \[u_3+v_1+w_2=u_6+v_4+w_5.\]

Let $q\in\Lambda$, so $q_1+q_2+q_3=q_4+q_5+q_6$. For each $1\le j\le 6$, let $(u_j,v_j,w_j)$ vary over the $(2d-3)$-dimensional affine subspace of $T_{x_j}M\times T_{y_j}M\times T_{z_j}M$ given by \begin{equation}\label{triplesum}u_j+v_j+w_j=q_j.\end{equation} 
Set 
\begin{eqnarray}\label{relations}&p_1=u_1+w_1,\,p_2=u_2+v_2,\,p_3=v_3+w_3,&\\
\label{relations2}&p_4=u_4+w_4,\,p_5=u_5+v_5,\,p_6=v_6+w_6.&\end{eqnarray} Since $\overline{\mathcal{G}}(y_1)\ne\overline{\mathcal{G}}(x_1),\,\overline{\mathcal{G}}(z_1)$, condition (\ref{triplesum}) implies that $p_1$ may be varied freely in $(q_1-T_{y_1}M)\cap(T_{x_1}M+T_{z_1}M)=q_1+T_{y_1}M$ and then $v_1=q_1-p_1$ is determined. Similarly, $p_2$ may be varied freely over a certain translate of $T_{z_2}M$ and then $w_2=q_2-p_2$ is determined, and so on for $p_3$, $p_4$, $p_5$, $p_6$ varying over translates of $T_{x_3}M$, $T_{y_4}M$, $T_{z_5}M$, $T_{x_6}M$ which then determine $u_3$, $v_4$, $w_5$, $u_6$ respectively. Since \[T_{y_1}M+T_{z_2}M+T_{x_3}M-T_{y_4}M-T_{z_5}M-T_{x_6}M=\reals^d,\] as $p_1,\ldots,p_6$ vary, the vector $p_1+p_2+p_3-p_4-p_5-p_6$ varies over all of $\reals^d$. Fix an admissible set of $p_1,\ldots,p_6$ such that 
\begin{equation}\label{psum}p_1+p_2+p_3=p_4+p_5+p_6.\end{equation} Since $\overline{\mathcal{G}}(x_1)\ne\overline{\mathcal{G}}(z_1)$, the relation (\ref{relations}) implies that we may vary $u_1$ freely in a certain translate of the linear subspace $T_{x_1}M\cap T_{z_1}M$ and then $w_1=p_1-u_1$ is determined, and so on for $v_2$, $w_3$, $u_4$, $v_5$, $w_6$ varying over translates of $T_{y_2}M\cap T_{x_2}M$, $T_{z_3}M\cap T_{y_3}M$, $T_{x_4}M\cap T_{z_4}M$, $T_{y_5}M\cap T_{x_5}M$, $T_{z_6}M\cap T_{y_6}M$ which then determine $u_2$, $v_3$, $w_4$, $u_5$, $v_6$ respectively. By (\ref{intersections}), it follows that \[T_{x_1}M\cap T_{z_1}M+T_{y_2}M\cap T_{x_2}M+T_{z_3}M\cap T_{y_3}M\]\[-\,T_{x_4}M\cap T_{z_4}M-T_{y_5}M\cap T_{x_5}M-T_{z_6}M\cap T_{y_6}M\]
\[=T_{x_1}M\cap T_{z_1}M+T_{y_2}M\cap T_{x_2}M+T_{z_3}M\cap T_{y_3}M\]
\[=\reals^d.\] Therefore as $u_1$,$v_2$, $w_3$, $u_4$, $v_5$, $w_6$ vary, the vector $u_1+v_2+w_3-u_4-v_5-w_6$ varies over all of $\reals^d$; choose admissible $u_1$,$v_2$, $w_3$, $u_4$, $v_5$, $w_6$ such that \begin{equation}\label{uvwsum}u_1+v_2+w_3=u_4+v_5+w_6.\end{equation}

By (\ref{psum}) and (\ref{uvwsum}), 
\[u_2+v_3+w_1=(p_2-v_2)+(p_3-w_3)+(p_1-u_1)\]\[=(p_4+p_5+p_6)-(u_4+v_5+w_6)=u_5+v_6+w_4\]
and then by (\ref{triplesum}),
\[u_3+v_1+w_2=(q_3-p_3)+(q_1-p_1)+(q_2-p_2)\]\[=(q_4+q_5+q_6)-(p_4+p_5+p_6)=u_6+v_4+w_5.\]
\end{proof}

Suppose that the functions $f_1$, $f_2$, $f_3$, $F$ above are measurable, that $f_1$, $f_2$, $f_3$ vanish only on a $\mu$-null set and that $\mu^3(M^3\backslash S)=0$. Similarly to \S\ref{section:strategy}, by Proposition~\ref{prop:gamma3surj} it follows that for each $z\in R_M$, there exists a non-empty open ball $T_z\subset\reals^d$ with center at $z$, translates $B_z$ and $C_z$ of $T_z$ and a measurable function $H:T_z+B_z+C_z\to\complex$ such that \[H(w_1+w_2+w_3)=F(w_1)F(w_2)F(w_3)\] for $\lambda^{3d}$-almost every triple $(w_1,w_2,w_3)\in T_z\times B_z\times C_z$. By fixing a typical $w_3$ and applying Lemma~\ref{lemma:classical}, it follows that there exist $c_z\in\complex\backslash\{0\}$ and $v_z\in\complex^d$ such that $F(w)=c_z\exp(v_z\cdot w)$ for $\lambda^d$-almost every $w\in T_z$. 

Define an equivalence relation $\sim$ on $R_M$ by declaring $w \sim u$ whenever $v_w = v_u$. This partitions $R_M$ into a collection of open sets $\{R_j\}_{j\in J}$ such that each $z\in R_j$ has a neighbourhood on which $F$ is of the form $F(z) = c\exp(v_j\cdot z)$, up to a null set, for some $c$ depending on the neighbourhood. 

For $j\in J$, consider the set $M_j$ of all $x\in M$ for which there exist $y,z\in M$ such that $x+y+z\in R_j$. Since $M$ is nowhere-flat, the sets $\{M_j\}_{j\in J}$ are an open cover for $M$. By Lemma~\ref{lemma:localise}, they form a partition for $M$. Assuming $M$ is connected, it follows that $M_j=M$ for some $j\in J$. This means that each $x\in M$ has an open neighbourhood $U$ in $M$ on which $f_1$ is of the form $f_1(x) = c_U\exp(v_j\cdot x)$ except for a $\mu$-null set for some constant $c_U\ne 0$. Since $M$ is connected, it follows that there is a constant $c_1\ne 0$ such that $f_1(x) = c_1\exp(v_j\cdot x)$ for $\mu$-almost every all $x\in M$. Thus the proof of Theorem~\ref{theorem:threefold} is complete for the case when $d\ge 3$ and $\mathcal{G}(M)$ is not contained in a planar circle.

Suppose now that $M\subset\reals^d$ is any nowhere-flat hypersurface. The case $d=2$ in Theorem~\ref{theorem:threefold} follows immediately from Theorem~\ref{theorem:curves}. For $d\ge 3$, each point of $x\in M$ has a small open neighbourhood $U$ such that either $U$ is a cylinder or $\mathcal{G}(U)$ is not contained in a planar circle. By Lemma~\ref{lemma:localise}, the following lemma therefore completes the proof of Theorem~\ref{theorem:threefold}. The additive analogue of the theorem is, once again, similar.

\begin{lemma}
Let $d\ge 3$. Suppose that $\Gamma\subset\reals^2$ is a nowhere-flat bounded smooth embedded planar curve with finite induced measure $\gamma$ and $V\subset\reals^{d-2}$ is a non-empty open ball. Then the nowhere-flat hypersurface $M=\Gamma\times V\subset\reals^d$ satisfies the following property. Suppose that $f_1, f_2, f_3:M\to\complex$ and $F:3M\to\complex$ are measurable such that $f_1$, $f_2$, $f_3$ vanish only on a $\gamma\times\lambda^{d-2}$-null set and for $(\gamma\times\lambda^{d-2})^3$-almost every $(x_1,x_2,x_3)\in M^3$, \[f_1(x_1)f_2(x_2)f_3(x_3)=F(x_1+x_2+x_3).\] Then there exist $c\in\complex\backslash\{0\}$ and $v\in\complex^d$ such that \[f_1(x)=c\exp(v\cdot x)\text{ for }\gamma\times\lambda^{d-2}\text{-a.e. }x\in M.\]  
\end{lemma}

\begin{proof}
Write points in $M$ as $(t,y)$ for $t\in \Gamma$, $y\in V$. For $1\le j\le 3$ and each $t\in\Gamma$, define $f_j^t:V\to\complex$ by $f_j^t(y)=f_j(t,y)$ and for each $s\in 3\Gamma$ define $F^s:3V\to\complex$ by $F^s(z)=F(s,z)$. Then for $\gamma^3$-almost every $(t_1,t_2,t_3)\in \Gamma^3$, the functions $f_j^{t_j}$ for $1\le j\le 3$ and $F^{t_1+t_2+t_3}$ are measurable, vanish only on a $\lambda^{d-2}$-null set and satisfy \[f_1^{t_1}(y_1)f_2^{t_2}(y_2)f_3^{t_3}(y_3)=F^{t_1+t_2+t_3}(y_1+y_2+y_3)\] for $\lambda^{3(d-2)}$-almost every $(y_1,y_2,y_3)\in V^3$.

By Lemma~\ref{lemma:classical}, it follows that for $1\le j\le 3$ and $\gamma^3$-almost every $(t_1,t_2,t_3)\in \Gamma^3$, \begin{equation}\label{fform}f_j^{t_j}(y)=c_{j,t_j}\exp(v_{t_1+t_2+t_3}\cdot y)\text{ for }\lambda^{d-2}\text{-almost every }y\in V\end{equation} and \[F^{t_1+t_2+t_3}(z)=c_{t_1+t_2+t_3}\exp(v_{t_1+t_2+t_3}\cdot z)\text{ for }\lambda^{3(d-2)}\text{-almost every }z\in 3V,\] where $c_{j,t_j},c_{t_1+t_2+t_3}\in\complex$ and $v_{t_1+t_2+t_3}\in\complex^{d-2}$. 

Thus for $\gamma^3$-almost every $(t_1,t_2,t_3)\in \Gamma^3$, \[c_{1,t_1}c_{2,t_2}c_{3,t_3}=c_{t_1+t_2+t_3}.\] Since $f_j$ is measurable it follows that the function $t\mapsto c_{j,t}$ defined $\gamma$-almost everywhere is measurable for each $1\le j\le 3$. Since $F$ is measurable and $\Gamma$ is nowhere-flat it follows that the functions $s\mapsto c_s$ and $s\mapsto v_s$ defined $\lambda^2$-almost everywhere on an open subset of $\reals^2$ which is dense in $3\Gamma$ are measurable. Since for each $1\le j\le 3$ the left-hand side of (\ref{fform}) is independent of $t_i,t_k$ for $\{i,k\}=\{1,2,3\}\backslash\{j\}$, it follows that for $\gamma^3$-almost every $(t_1,t_2,t_3)\in\Gamma^3$, the complex vector $v_{t_1+t_2+t_3}=v'$ is constant. By Theorem~\ref{theorem:curves}, it follows that there exist $b_j\in\complex$ for $1\le j\le 3$ and $v''\in\complex^2$ such that $c_{j,t}=b_j\exp(v''\cdot t)$ for $\gamma$-almost every $t\in\Gamma$ for each $1\le j\le 3$. Thus \[f_j(x)=b_j\exp(v\cdot x)\text{ for }\gamma\times\lambda^{d-2}\text{-almost every }x\in M,\] where $v=(v'',v')\in\complex^d$.
\end{proof}


\begin{thebibliography}{20}

\bibitem{aczel}
J.~Aczel,
{\em Lectures on Functional Equations and Their Applications}, 
Academic Press (1966).

\bibitem{christshao}
M.~Christ and S.~Shao,
{\em On the extremizers of an adjoint Fourier restriction inequality}, Adv. in Math. 230 (2012), no.3, 957-977.

\bibitem{christyoungs}
M.~Christ,
{\em Near-extremizers of Young's inequality for $\reals^d$}, pre-print, arXiv:1112.4875.

\bibitem{foschi}
D.~Foschi,
{\em Maximizers for the Strichartz Inequality},
J. Eur. Math. Soc. 9 (2007), no. 4.

\bibitem{ponomarev}
S.~P.~Ponomarev, {\em Submersions and preimages of sets of measure zero},
Siberian Math. Journal 28 (1987), 153-163.

\bibitem{regularity}
A.~Jarai, {\em Regularity Properties of Functional Equations in Several Variables},
Adv. in Math., Vol. 8 (2005).

\end{thebibliography}
\end{document}